\documentclass[12pt]{amsart}

\usepackage{amssymb}
\usepackage[all]{xy}

\newcommand{\PP}{\mathbb{P}}
\newcommand{\RR}{\mathbb{R}}
\newcommand{\CC}{\mathbb{C}}
\newcommand{\ZZ}{\mathbb{Z}}
\newcommand{\HH}{\mathbb{H}}

\newcommand{\fg}{\mathfrak{g}}

\newcommand{\fl}{\mathfrak{l}}
\newcommand{\fc}{\mathfrak{c}}

\newcommand{\calO}{\mathcal{O}}
\newcommand{\calL}{\mathcal{L}}
\newcommand{\calM}{\mathcal{M}}
\newcommand{\Bun}{\calM}
\newcommand{\Higgs}{\calM_{\mathrm{Higgs}}}
\newcommand{\Rep}{\calM_{\mathrm{Rep}}}
\newcommand{\Conn}{\calM_{\mathrm{conn}}}
\newcommand{\Hod}{\calM_{\mathrm{Hod}}}
\newcommand{\MDH}{\calM_{\mathrm{DH}}}
\newcommand{\Xbar}{\overline{X}}

\newcommand{\id}{\mathrm{id}}
\newcommand{\pr}{\mathrm{pr}}
\newcommand{\regstab}{\mathrm{rs}}
\newcommand{\rmH}{\mathrm{H}}

\DeclareMathOperator{\ad}{ad}
\DeclareMathOperator{\At}{At}
\DeclareMathOperator{\degree}{degree}
\DeclareMathOperator{\rank}{rank}
\DeclareMathOperator{\Hom}{Hom}

\DeclareMathOperator{\Aut}{Aut}
\DeclareMathOperator{\SL}{SL}
\DeclareMathOperator{\Sym}{Sym}
\DeclareMathOperator{\Lie}{Lie}

\newcommand{\longto}[1][]{\stackrel{#1}{\longrightarrow}}
\newcommand{\longonto}{\relbar\joinrel\twoheadrightarrow}
\newcommand{\longhookrightarrow}[1][]
  {\stackrel{#1}{\lhook\joinrel\relbar\joinrel\rightarrow}}
\newcommand{\longhookleftarrow}[1][]
  {\stackrel{#1}{\leftarrow\joinrel\relbar\joinrel\rhook}}

\theoremstyle{plain}
\newtheorem{theorem}{Theorem}[section]
\newtheorem{proposition}[theorem]{Proposition}

\newtheorem{corollary}[theorem]{Corollary}
\theoremstyle{definition}

\theoremstyle{remark}

\begin{document}

\title[Deligne--Hitchin moduli space, II]{Torelli
theorem for the\\Deligne--Hitchin moduli space, II}

\thanks{This work was supported by the Spanish
Ministerio de Ciencia e Innovaci\'on (grant MTM2010-17389),
and by the German SFB 647: Raum - Zeit - Materie.}

\author[I. Biswas]{Indranil Biswas}

\address{School of Mathematics, Tata Institute of Fundamental
Research, Homi Bhabha Road, Bombay 400005, India}

\email{indranil@math.tifr.res.in}

\author[T. L. G\'omez]{Tom\'as L. G\'omez}

\address{Instituto de Ciencias Matem\'aticas (CSIC-UAM-UC3M-UCM),
Nicol\'as Cabrera 15, Campus Cantoblanco UAM, 28049 Madrid, Spain}

\email{tomas.gomez@icmat.es}

\author[N. Hoffmann]{Norbert Hoffmann}

\address{Freie Universit\"at Berlin, Institut f\"ur Mathematik, 
Arnimallee 3, 14195 Berlin, Germany}

\email{norbert.hoffmann@fu-berlin.de}



\date{}

\begin{abstract}
Let $X$ and $X'$ be compact Riemann surfaces of genus at least three.
Let $G$ and $G'$ be nontrivial connected semisimple linear algebraic
groups over $\CC$. If some components $\MDH^d(X,G)$ and
$\MDH^{d'}(X',G')$ of the associated Deligne--Hitchin moduli
spaces are biholomorphic, then $X'$ is isomorphic to $X$ or
to the conjugate Riemann surface $\Xbar$.
\end{abstract}

\maketitle

\section{Introduction}

Let $X$ be a compact connected Riemann surface of genus
$g \geq 3$. Let $\Xbar$ denote the conjugate Riemann surface;
by definition, it consists of the real manifold underlying $X$
and the almost complex structure $J_{\Xbar} := -J_X$. Let $G$ be a
nontrivial connected semisimple linear algebraic group over $\CC$.
The topological types of holomorphic principal $G$--bundles $E$
over $X$ correspond to elements of $\pi_1(G)$. Let
$\Higgs^d( X, G)$ denote the moduli space of semistable Higgs
$G$--bundles $(E, \theta)$ over $X$ with $E$ of topological type
$d\,\in\, \pi_1(G)$.

The \emph{Deligne--Hitchin moduli space} \cite{Si3} is a complex
analytic space $\MDH^d( X, G)$ associated to $X$, $G$ and $d$.
It is the twistor space for the hyper-K\"ahler structure on
$\Higgs^d( X, G)$; see \cite[\S 9]{Hi2}. Deligne \cite{De} has
constructed it together with a surjective holomorphic map
\begin{equation*}
  \MDH^d( X, G) \longonto \CC\PP^1 = \CC \cup \{ \infty \}.
\end{equation*}
The inverse image of $\CC \subseteq \CC\PP^1$ is the moduli space
$\Hod^d( X, G)$ of holomorphic principal $G$--bundles over $X$
endowed with a $\lambda$--connection. In particular,
every fiber over $\CC^* \subset \CC\PP^1$ is isomorphic to the
moduli space of holomorphic $G$--connections over $X$. The fiber
over $0 \in \CC\PP^1$ is $\Higgs^d( X, G)$, and the fiber over
$\infty \in \CC\PP^1$ is $\Higgs^{-d}( \Xbar, G)$.

In this paper, we study the dependence of these moduli spaces
on $X$. Our main result, Theorem \ref{thm:MDH}, states
that the complex analytic space $\MDH^d( X, G)$ determines
the unordered pair $\{X, \Xbar\}$ up to isomorphism.
We also prove that $\Higgs^d( X, G)$ and $\Hod^d( X, G)$
each determine $X$ up to isomorphism;
see Theorem \ref{thm:higgs} and Theorem \ref{thm:hodge}.

The key technical result is Proposition \ref{prop:fixed_higgs},
which says the following: Let $Z$ be an irreducible component of
the fixed point locus for the natural $\CC^*$--action on a moduli
space $\Higgs^d( X, G)$ of Higgs $G$--bundles. Then,
  \begin{equation*}
    \dim Z \leq (g-1) \cdot \dim_{\mathbb C} G,
  \end{equation*}
  with equality holding only for $Z = \Bun^d( X, G)$.

In \cite{BGHL}, the case of $G = \SL( r, \CC)$ was considered.

\section{Some moduli spaces associated to a compact
Riemann surface} \label{sec:moduli}
Let $X$ be a compact connected Riemann surface of genus
$g \geq 3$. Let $G$ be a nontrivial connected semisimple linear
algebraic group defined over $\CC$, with Lie algebra $\fg$.

\subsection{Principal $G$--bundles}
We consider holomorphic principal $G$--bundles $E$ over $X$.
Recall that the topological type of $E$ is given by an element
$d \in \pi_1(G)$ \cite{Ra}; this is a finite abelian group. The
\emph{adjoint vector bundle} of $E$ is the holomorphic vector
bundle
\begin{equation*}
  \ad( E) := E \times^G \fg
\end{equation*}
over $X$, using the adjoint action of $G$ on $\fg$.
$E$ is called \emph{stable} (respectively, \emph{semistable}) if
\begin{equation} \label{eq:stability}
  \degree( \ad( E_P)) < 0 \quad (\text{respectively, } \leq 0)
\end{equation}
for every maximal parabolic subgroup $P \subsetneqq G$ and every
holomorphic reduction of structure group $E_P$ of $E$ to $P$; here
$\ad( E_P) \subset \ad( E)$ is the adjoint vector bundle of $E_P$.

Let $\Bun^d( X, G)$ denote the moduli space of semistable
holomorphic principal $G$--bundles $E$ over $X$ of topological
type $d \in \pi_1(G)$. It is known that $\Bun^d( X, G)$ is an
irreducible normal projective variety of dimension
$(g-1) \cdot \dim_{\CC} G$ over $\CC$.

\subsection{Higgs $G$--bundles}
The holomorphic cotangent bundle of $X$ will be denoted by $K_X$.

A \emph{Higgs $G$--bundle} over $X$ is a pair $( E, \theta)$
consisting of a holomorphic principal $G$--bundle $E$ over $X$
and a holomorphic section
\begin{equation*}
  \theta \in \rmH^0( X, \ad( E) \otimes K_X),
\end{equation*}
the so-called \emph{Higgs field} \cite{Hi1, Si1}.
The pair $(E, \theta)$ is called \emph{stable} (respectively,
\emph{semistable}) if the inequality \eqref{eq:stability} holds
for every holomorphic reduction of structure group $E_P$ of $E$
to a maximal parabolic subgroup $P \subsetneqq G$ such that
$\theta \in \rmH^0( X, \ad( E_P) \otimes K_X)$.

Let $\Higgs^d( X, G)$ denote the moduli space of semistable
Higgs $G$--bundles $( E, \theta)$ over $X$ such that $E$ is of
topological type $d \in \pi_1( G)$. It is known that
$\Higgs^d( X, G)$ is an irreducible normal quasiprojective
variety of dimension $2( g-1) \cdot \dim_{\CC} G$ over $\CC$
\cite{Si2}. We regard $\Bun^d( X, G)$ as a closed subvariety
of $\Higgs^d( X, G)$ by means of the embedding
\begin{equation*}
  \Bun^d( X, G) \,\longhookrightarrow\, \Higgs^d( X, G)\, ,
    \qquad E \,\longmapsto\, (E, 0)\, .
\end{equation*}
There is a natural algebraic symplectic structure
on $\Higgs^d( X, G)$; see \cite{Hi1, BR}.

\subsection{Representations of the surface group in $G$}
Fix a base point $x_0 \in X$. The fundamental group of $X$
admits a standard presentation
\begin{equation*}
  \pi_1( X, x_0) \,\cong\, \langle a_1, \ldots, a_g, b_1, \ldots, b_g
    | \prod_{i=1}^g a_i b_i a_i^{-1} b_i^{-1} = 1 \rangle
\end{equation*}
which we choose in such a way that it is compatible with the
orientation of $X$. We identify the fundamental group of $G$
with the kernel of the universal covering $\tilde{G} \longonto G$.
The \emph{type} $d \in \pi_1( G)$ of a homomorphism
$\rho: \pi_1( X, x_0) \longto G$ is defined by
\begin{equation*}
  d \,:=\, \prod_{i=1}^g \alpha_i \beta_i \alpha_i^{-1} \beta_i^{-1}
    \,\in\, \pi_1( G) \,\subset\, \tilde{G}
\end{equation*}
for any choice of lifts $\alpha_i, \beta_i \in \tilde{G}$ of
$\rho( a_i), \rho( b_i) \in G$. This is also the topological
type of the principal $G$--bundle $E_{\rho}$ over $X$ given by
$\rho$. The space $\Hom^d( \pi_1( X, x_0), G)$ of all
homomorphisms $\rho: \pi_1( X, x_0) \longto G$ of type
$d \in \pi_1( G)$ is an irreducible affine variety over
$\CC$, and $G$ acts on it by conjugation. The GIT quotient
\begin{equation*}
  \Rep^d( X, G) := \Hom^d( \pi_1( X, x_0), G)/\!\!/G
\end{equation*}
doesn't depend on $x_0$. It is an affine
variety of dimension $2( g-1) \cdot \dim_{\CC} G$
over $\CC$, which carries a natural symplectic form
\cite{AB, Go}. Its points represent equivalence
classes of completely reducible homomorphisms $\rho$.
There is a natural bijective map
\begin{equation*}
  \Rep^d( X, G) \longto \Higgs^d( X, G)
\end{equation*}
given by a variant of the Kobayashi--Hitchin correspondence
\cite{Si1}. This bijective map is not holomorphic.

\subsection{Holomorphic $G$--connections}
Let $p: E \longto X$ be a holomorphic principal $G$--bundle.
Because the vertical tangent space at every point of the
total space $E$ is canonically isomorphic to $\fg$, there is
a natural exact sequence
\begin{equation*}
  0 \longto E \times \fg \longto TE \longto[dp] p^* TX \longto 0
\end{equation*}
of $G$-equivariant holomorphic vector bundles over $E$.
Taking the $G$-invariant direct image under $p$, it follows that
the \emph{Atiyah bundle} for $E$
\begin{equation*}
  \At( E) := p_*( TE)^G \subset p_*( TE)
\end{equation*}
sits in a natural exact sequence of holomorphic vector bundles
\begin{equation} \label{eq:atiyah}
  0 \longto \ad( E) \longto \At( E) \longto[dp] TX \longto 0
\end{equation}
over $X$. This exact sequence is called the
\emph{Atiyah sequence}. A \emph{holomorphic connection} on $E$
is a splitting of the Atiyah sequence, or in other words
a holomorphic homomorphism
\begin{equation*}
  D: TX \longto \At( E)
\end{equation*}
such that $dp \circ D = \id_{TX}$. It always exists if $E$ is
semistable \cite{At, AzBi}. The curvature of $D$ is a holomorphic
2--form with values in $\ad( E)$, so $D$ is automatically flat.

A \emph{holomorphic $G$--connection} is a pair $( E, D)$ where
$E$ is a holomorphic principal $G$--bundle over $X$, and $D$ is
a holomorphic connection on $E$. Such a pair is automatically
semistable, because the degree of a flat vector bundle is zero.

Let $\Conn^d( X, G)$ denote the moduli space of holomorphic 
$G$--connections $( E, D)$ over $X$ such that $E$ is of
topological type $d \in \pi_1( G)$. It is known that
$\Conn^d( X, G)$ is an irreducible quasiprojective variety
of dimension $2( g-1) \cdot \dim_{\CC} G$ over $\CC$.

Sending each holomorphic $G$--connection to its monodromy
defines a map
\begin{equation} \label{eq:Riemann-Hilbert}
  \Conn^d( X, G) \longto \Rep^d( X, G)
\end{equation} 
which is biholomorphic, but not algebraic; it is called
Riemann--Hilbert correspondence. The inverse map sends
a homomorphism $\rho: \pi_1( X, x_0) \longto G$ to the
associated principal $G$--bundle $E_{\rho}$, endowed with
the induced holomorphic connection $D_{\rho}$.

\subsection{$\lambda$--connections}
Let $p: E \longto X$ be a holomorphic principal $G$--bundle.
For any $\lambda \in \CC$, a \emph{$\lambda$--connection}
on $E$ is a holomorphic homomorphism of vector bundles
\begin{equation*}
  D : TX \longto \At( E)
\end{equation*}
such that $dp \circ D = \lambda \cdot \id_{TX}$ for the
epimorphism $dp$ in the Atiyah sequence \eqref{eq:atiyah}.
Therefore, a $0$--connection is a Higgs field,
and a $1$--connection is a holomorphic connection.

If $D$ is a $\lambda$--connection on $E$ with $\lambda \neq 0$,
then $\lambda^{-1} D$ is a holomorphic connection on $E$.
In particular, the pair $(E, D)$ is automatically semistable
in this case.

Let $\Hod^d( X, G)$ denote the moduli space of triples
$(\lambda, E, D)$, where $\lambda \in \CC$, $E$ is a
holomorphic principal $G$--bundle over $X$ of topological
type $d \in \pi_1( G)$, and $D$ is a semistable
$\lambda$--connection on $E$; see \cite{Si2}.
There is a canonical algebraic map
\begin{equation} \label{eq:pr_hodge}
  \pr = \pr_X: \Hod^d( X, G) \longonto \CC,
    \qquad ( \lambda, E, D) \longmapsto \lambda.
\end{equation}
Its fibers over $\lambda = 0$ and $\lambda = 1$ are
$\Higgs^d( X, G)$ and $\Conn^d( X, G)$, respectively.
The Riemann--Hilbert correspondence
\eqref{eq:Riemann-Hilbert} allows to define a holomorphic
open embedding
\begin{equation*}
  j = j_X: \CC^* \times \Rep^d( X, G) \longhookrightarrow
    \Hod^d( X, G), \qquad ( \lambda, \rho) \longmapsto
      ( \lambda, E_{\rho}, \lambda D_{\rho})
\end{equation*}
with image $\pr^{-1}( \CC^*)$. This map commutes with the
projections onto $\CC^*$.

\subsection{The Deligne--Hitchin moduli space}
The compact Riemann surface $X$ provides an underlying
real $C^\infty$ manifold $X_{\RR}$, and an almost complex
structure $J_X: TX_{\RR} \longto TX_{\RR}$. Since any
almost complex structure in real dimension two is integrable, 
\begin{equation*}
  \Xbar := ( X_{\RR}, -J_X)
\end{equation*}
is a compact Riemann surface as well. It has the opposite
orientation, so
\begin{equation} \label{eq:Reps}
  \Rep^d( X, G) = \Rep^{-d}( \Xbar, G).
\end{equation}
The \emph{Deligne--Hitchin moduli space} $\MDH^d( X, G)$
is the complex analytic space obtained by gluing
$\Hod^d( X, G)$ and $\Hod^{-d}( \Xbar, G)$ along their
common open subspace
\begin{equation*}
  \Hod^d( X, G) \longhookleftarrow[j_X]
    \CC^* \times \Rep^d( X, G) \cong
    \CC^* \times \Rep^{-d}( \Xbar, G)
    \longhookrightarrow[j_{\Xbar}] \Hod^{-d}( \Xbar, G)
\end{equation*}
where the isomorphism in the middle sends $( \lambda, \rho)$
to $( 1/\lambda, \rho)$; see \cite{Si3,De}.
The projections $\pr_X$ on $\Hod^d( X, G)$ and
$1/\pr_{\Xbar}$ on $\Hod^{-d}( \Xbar, G)$ patch together
to a holomorphic map
\begin{equation*}
  \MDH^d( X, G) \longonto \CC\PP^1 = \CC \cup \{ \infty \}.
\end{equation*}
Its fiber over any $\lambda \in \CC^*$ is biholomorphic to
the representation space \eqref{eq:Reps}, whereas its fibers
over $\lambda = 0$ and $\lambda = \infty$ are 
$\Higgs^d( X, G)$ and $\Higgs^{-d}( \Xbar, G)$, respectively.

\section{Fixed points of the natural $\CC^*$--action}
\label{sec:action}
The group $\CC^*$ acts algebraically on the moduli space
$\Higgs^d( X, G)$, via the formula
\begin{equation} \label{eq:action_higgs}
  t \cdot ( E, \theta) := ( E, t \theta).
\end{equation}
The fixed point locus $\Higgs^d( X, G)^{\CC^*}$ contains
the closed subvariety $\Bun^d( X, G)$.
\begin{proposition} \label{prop:fixed_higgs}
  Let $Z$ be an irreducible component of
  $\Higgs^d( X, G)^{\CC^*}$. Then one has
  \begin{equation*}
    \dim Z \leq (g-1) \cdot \dim_{\mathbb C} G,
  \end{equation*}
  with equality holding only for $Z = \Bun^d( X, G)$.
\end{proposition}
\begin{proof}
  Let $( E, \theta)$ be a stable Higgs $G$--bundle over $X$.
  Its infinitesimal deformations are, according to
  \cite[Theorem 2.3]{BR}, governed by the complex of
  vector bundles
  \begin{equation} \label{eq:complex}
    C^0 \,:=\, \ad( E) \,\xrightarrow{\ad( \theta)}\,
      \ad( E) \otimes K_X \,=: \,C^1
  \end{equation}
  over $X$. Since $( E, \theta)$ is stable, it has no
  infinitesimal automorphisms, so
  \begin{equation*}
    \HH^0( X, C^{\bullet}) \,=\, 0\, .
  \end{equation*}
  The Killing form on $\fg$ induces isomorphisms $\fg^* \cong \fg$
  and $\ad( E)^* \cong \ad( E)$. Hence the vector bundle $\ad( E)$
  has degree $0$. Serre duality allows us to conclude
  \begin{equation*}
    \HH^2( X, C^{\bullet}) \,=\, 0\, .
  \end{equation*}
  Using all this, the Riemann--Roch formula yields
  \begin{equation} \label{eq:dim}
    \dim \HH^1( X, C^{\bullet}) \,=\, 2( g-1) \cdot \dim_{\CC} G\, .
  \end{equation}

  From now on, we assume that the point $( E, \theta)$
  is fixed by $\CC^*$, and we also assume $\theta \neq 0$.
  Then $( E, \theta) \cong ( E, t \theta)$ for all $t \in \CC^*$,
  so the sequence of complex algebraic groups
  \begin{equation*}
    1 \longto \Aut( E, \theta) \longto \Aut( E, \CC \theta)
      \longto \Aut( \CC \theta) = \CC^* \longto 1
  \end{equation*}
  is exact. Because $( E, \theta)$ is stable, $\Aut( E, \theta)$
  is finite. Consequently, the identity component of
  $\Aut(E, \CC \theta)$ is isomorphic to $\CC^*$.
  This provides an embedding
  \begin{equation*}
    \CC^* \longhookrightarrow \Aut( E)\, , \qquad
      t \longmapsto \varphi_t\, ,
  \end{equation*}
  and an integer $w \neq 0$ with $\varphi_t( \theta)\,=\,t^w \cdot
  \theta$ for all $t \in \CC^*$. We may assume that $w \geq 1$.

  Choose a point $e_0 \in E$.
  Then there is a unique group homomorphism
  \begin{equation*}
    \iota\,:\, \CC^* \,\longto\, G
  \end{equation*}
  such that $\varphi_t( e_0) = e_0 \cdot \iota( t)$ for all
  $t \in \CC^*$. The conjugacy class of $\iota$ doesn't depend
  on $e_0$, since the space of conjugacy classes
  $\Hom( \CC^*, G)/G$ is discrete. The subset
  \begin{equation*}
    E_H \,:=\, \{ e \in E: \varphi_t( e) = e \cdot \iota( t)
      \text{ for all } t \in \CC^*\}
  \end{equation*}  
  of $E$ is a holomorphic reduction of structure group
  to the centralizer $H$ of $\iota( \CC^*)$ in $G$. Let
  \begin{equation} \label{eq:grading}
    \fg \,=\, \bigoplus_{n \in \ZZ} \fg_n
  \end{equation}
  denote the eigenspace decomposition given by the adjoint action
  of $\CC^*$ on $\fg$ via $\iota$. 

  Let $N \in \ZZ$ be maximal with $\fg_N \neq 0$.   
  Let $P \subset G$ be the parabolic subgroup with
  \begin{equation*}
    \Lie( P) \,=\, \bigoplus_{n \geq 0} \fg_n \subset \fg\, .
  \end{equation*}
  Since $H \subset G$ has Lie algebra $\fg_0$,
  it is a Levi subgroup in $P$. Choose subgroups
  \begin{equation*}
    \iota( \CC^*) \,\subseteq\, T \,\subseteq \,
B \,\subseteq\, P \,\subset\, G
  \end{equation*}
  such that $T$ is a maximal torus in $H \subset G$
  and $B$ is a Borel subgroup in $G$. Let
  \begin{equation*}
    \alpha_j: T \longto \CC^* \qquad\text{and}\qquad
    \alpha_j^{\vee}: \CC^* \longto T
  \end{equation*}
  be the resulting simple roots and coroots of $G$. We denote by
  $\langle \_, \_ \rangle$ the natural pairing between characters
  and cocharacters of $T$. Let $\alpha_j$ be a simple root of $G$
  with $\langle \alpha_j, \iota \rangle > 0$, and let $\beta$ be
  a root of $G$ with $\langle \beta, \iota \rangle = N$.
  Then the elementary reflection
  \begin{equation*}
    s_j( \beta)\,
      =\, \beta - \langle \beta, \alpha_j^{\vee} \rangle \alpha_j
  \end{equation*}
  is a root of $G$, so $\langle s_j(\beta), \iota \rangle \leq N$;
  this implies that $\langle \beta, \alpha_j^{\vee} \rangle\,\geq\, 0$.
  The sum of all such roots $\beta$ with
  $\langle \beta, \iota \rangle = N$ is the restriction
  $\chi|_T$ of the determinant
  \begin{equation} \label{eq:character}
    \chi: P \longto \Aut( \fg_N) \xrightarrow{\det} \CC^*
  \end{equation}
  of the adjoint action of $P$ on $\fg_N$. Hence we conclude
  $\langle \chi|_T, \alpha_j^{\vee} \rangle \geq 0$ for all simple
  roots $\alpha_j$ with $\langle \alpha_j, \iota \rangle > 0$.
  This means that the character $\chi$ of $P$ is dominant.

  The decomposition \eqref{eq:grading} of $\fg$
  induces a vector bundle decomposition
  \begin{equation*}
    \ad( E)\,=\, \bigoplus_{n \in \ZZ} E_H \times^H \fg_n\, .
  \end{equation*}
  Since $\CC^*$ acts with weight $w$ on the Higgs field $\theta$
  by construction, we have
  \begin{equation} \label{eq:theta}
    \theta \,\in\,
      \rmH^0 \big( X, (E_H \times^H \fg_w) \otimes K_X \big)\, .
  \end{equation}
  In particular, $\theta \in \rmH^0( X, \ad( E_P) \otimes K_X)$
  for the reduction $E_P := E_H \times^H P \subseteq E$ of the
  structure group to $P$. The Higgs version of the stability
  criterion \cite[Lemma 2.1]{Ra} yields
  \begin{equation*}
    \degree( E_H \times^H \fg_N) \,\leq\, 0
  \end{equation*}
  since $P$ acts on $\det( \fg_N)$ via the dominant character
  $\chi$ in \eqref{eq:character}. Now Riemann--Roch implies
the following:
  \begin{equation} \label{eq:H1}
    \dim \rmH^1( X, E_H \times^H \fg_N)
      \,\geq\, (g-1) \cdot \dim_{\CC} \fg_N \,>\,  0\, .
  \end{equation}

  The complex $C^{\bullet}$ in \eqref{eq:complex} is, due to
  \eqref{eq:theta}, the direct sum of its subcomplexes
  $C^{\bullet}_n$ given by
  \begin{equation*}
    C^0_n\,:=\,E_H \times^H \fg_n\,\xrightarrow{\ad( \theta)}\,
      (E_H \times^H \fg_{n+w}) \otimes K_X \,=:\, C^1_n\, .
  \end{equation*}
  Thus the hypercohomology of $C^{\bullet}$ decomposes as well;
  in particular, we have
  \begin{equation*}
    \HH^1( X, C^{\bullet}) \,=\,
      \bigoplus_{n \in \ZZ} \HH^1( X, C^{\bullet}_n)\, .
  \end{equation*}
  In the last nonzero summand $C^{\bullet}_N$,
  we have $C^1_N = 0$ and hence
  \begin{equation*}
    \dim \HH^1(X,C^{\bullet}_N)
    \,=\, \dim \rmH^1( X, E_H \times^H \fg_N) \,>\, 0
  \end{equation*}
  due to \eqref{eq:H1}. Since $\fg_n^* \cong \fg_{-n}$ via the
  Killing form on $\fg$, Serre duality yields in particular
  \begin{equation*}
    \dim \HH^1(X,C^{\bullet}_0) = \dim \HH^1(X,C^{\bullet}_{-w}).
  \end{equation*}
  Taken together, the last three formulas and the equation
  \eqref{eq:dim} imply that
  \begin{equation*}
    \dim \HH^1(X,C^{\bullet}_0) \,<\, \frac{1}{2} \dim
      \HH^1( X,C^{\bullet}) \,=\,(g-1) \cdot \dim_{\CC} G\, .
  \end{equation*}
  But $\HH^1(X,C^{\bullet}_0)$ parameterizes infinitesimal
  deformations of pairs $(E_H, \theta)$ consisting of a principal
  $H$--bundle $E_H$ and a section $\theta$ as in \eqref{eq:theta};
  see \cite[Theorem 2.3]{BR}. This proves that
  \begin{equation*}
    \dim Z < (g-1) \cdot \dim_{\CC} G
  \end{equation*}
  for every irreducible component $Z$ of the fixed point locus
  $\Higgs^d( X, G)^{\CC^*}$ such that $Z$ contains stable Higgs
  $G$--bundles $( E, \theta)$ with $\theta \neq 0$.

  The non-stable points in $\Higgs^d( X, G)$ correspond to
  polystable Higgs $G$--bundles $( E, \theta)$. Polystability
  means that $E$ admits a reduction of structure group $E_L$
  to a Levi subgroup $L \subsetneqq G$ of a parabolic subgroup
  in $G$ such that $\theta$ is a section of the subbundle
  \begin{equation*}
    \ad( E_L) \otimes K_X \,\subset\, \ad( E) \otimes K_X
  \end{equation*}
  and the pair $( E_L, \theta)$ is stable.
  Let $C \subseteq L$ be the identity component of the center,
  and let $\fc \subseteq \fl$ be their Lie algebras. Then
  $E_{L/C} := E_L/C$ is a principal $(L/C)$--bundle over $X$, and
  \begin{equation*}
    \ad( E_L) \cong (\fc \otimes \calO_X) \oplus \ad( E_{L/C})
  \end{equation*}
  since $\fl = \fc \oplus [ \fl, \fl]$, where the subalgebra
  $[ \fl, \fl] \subseteq \fl$ is also the Lie algebra of $L/C$.
  We have
  \begin{equation*}
    \dim_{\CC} G - \dim_{\CC} L \,\geq\, 2 \dim_{\CC} C
  \end{equation*}
  because maximal Levi subgroups in $G$ have $1$-dimensional
  center and at least one pair of opposite roots less than $G$;
  the other Levi subgroups can be reached by iterating this.

  Now suppose that $\CC^*$ fixes the point $( E, \theta)$. Then
  $( E_L, \theta) \cong ( E_L, t \theta)$ for all $t \in \CC^*$.
  But the action of $\Aut( E_L)$ on the direct summand
  $\fc \otimes \calO_X$ of $\ad( E_L)$ is trivial, since the
  adjoint action of $L$ on $\fc$ is trivial. So $\theta$ lives
  in the other summand of $\ad( E_L)$, meaning
  \begin{equation*}
    \theta\,\in\,\rmH^0 \big( X, \ad( E_{L/C}) \otimes K_X \big)\, .
  \end{equation*}
  The Higgs $(L/C)$--bundle $( E_{L/C}, \theta)$ is still stable
  and fixed by $\CC^*$; we have already proved that the locus of
  such has dimension $\leq (g-1) \cdot \dim_{\CC} (L/C)$. The
  abelian variety $\Bun^0( X, C)$ acts simply transitively on
  lifts of $E_{L/C}$ to a principal $L$--bundle $E_L$, so these
  lifts form a family of dimension $g \cdot \dim_{\CC} C$.
  Hence the pairs $( E_L, \theta)$ in question have at most
  \begin{equation*}
    (g-1) \cdot \dim_{\CC} (L/C) + g \cdot \dim_{\CC} C
      < (g-1) \cdot \dim_{\CC} G
  \end{equation*}
  moduli. This implies that $\dim Z \,< \, (g-1) \cdot \dim_{\CC} G$
  for each non-stable component $Z$ of the fixed point locus,
  since there are only finitely many possibilities for $L$
  up to conjugation.
\end{proof}
The algebraic $\CC^*$--action \eqref{eq:action_higgs} on
$\Higgs^d( X, G)$ extends naturally to an algebraic
$\CC^*$--action on $\Hod^d( X, G)$, which is given by the formula
\begin{equation} \label{eq:action_hodge}
  t \cdot ( \lambda, E, D) := ( t \lambda, E, t D).
\end{equation}
A point $( \lambda, E, D)$ can only be fixed by this action if
$\lambda = 0$, so Proposition \ref{prop:fixed_higgs}
yields the following corollary:

\begin{corollary} \label{cor:fixed_hodge}
  Let $Z$ be an irreducible component of $\Hod^d( X, G)^{\CC^*}$.
  Then one has
  \begin{equation*}
    \dim Z \leq (g-1) \cdot \dim_{\mathbb C} G,
  \end{equation*}
  with equality only for $Z = \Bun^d( X, G)$.
\end{corollary}

The algebraic $\CC^*$--action \eqref{eq:action_hodge}
on $\Hod^d( X, G)$ extends naturally to a holomorphic
$\CC^*$--action on $\MDH^d( X, G)$, which is on the other
open patch $\Hod^{-d}( \Xbar, G)$ given by the formula
\begin{equation*}
  t \cdot (\lambda, E, D) := (t^{-1} \lambda, E, t^{-1} D).
\end{equation*}
Applying Corollary \ref{cor:fixed_hodge} to both $\Hod^d( X, G)$
and $\Hod^{-d}( \Xbar, G)$, one immediately gets
\begin{corollary} \label{cor:fixed_MDH}
  Let $Z$ be an irreducible component of $\MDH^d( X, G)^{\CC^*}$.
  Then one has
  \begin{equation*}
    \dim Z \leq (g-1) \cdot \dim_{\mathbb C} G,
  \end{equation*}
  with equality only for $Z = \Bun^d( X, G)$
  and for $Z = \Bun^{-d}( \Xbar, G)$.
\end{corollary}

\section{Vector fields on the moduli spaces}
\label{sec:vectorfields}
A stable principal $G$--bundle $E$ over $X$ is called
\emph{regularly stable} if the automorphism group $\Aut( E)$
is just the center of $G$. The regularly stable locus
\begin{equation*}
  \Bun^{d, \regstab}( X, G) \subseteq \Bun^d( X, G)
\end{equation*}
is open, and coincides with the smooth locus of $\Bun^d( X, G)$;
see \cite[Corollary 3.4]{BH}.
\begin{proposition} \label{prop:fields}
  There are no nonzero holomorphic vector fields
  on $\Bun^{d, \regstab}( X, G)$.
\end{proposition}
\begin{proof}
  This statement is contained in \cite[Corollary III.3]{Fa}.
\end{proof}
\begin{proposition} \label{prop:forms}
  There are no nonzero holomorphic $1$--forms
  on $\Bun^{d, \regstab}( X, G)$.
\end{proposition}
\begin{proof}
  The moduli space of Higgs $G$--bundles is
  equipped with the \emph{Hitchin map}
  \begin{equation*}
    \Higgs^d( X, G) \longto
      \bigoplus_{i=1}^{\rank( G)} H^0( X, K^{\otimes n_i}_X)
  \end{equation*}
  where the $n_i$ are the degrees of generators for the
  algebra $\Sym( \fg^*)^G$; see \cite[\S~4]{Hi1}, \cite{La}.
  Any sufficiently general fiber of this Hitchin map
  is a complex abelian variety $A$, and
  \begin{equation*} \xymatrix{
    \varphi: A \ar@{-->}[r] & \Bun^{d, \regstab}( X, G),
      & (E, \theta) \longmapsto E,
  } \end{equation*}
  is a dominant rational map. This rational map $\varphi$ is
  defined outside a closed subscheme of codimension at least two;
  see \cite[Theorem II.6]{Fa}.

  Let $\omega$ be a holomorphic $1$--form on
  $\Bun^{d, \regstab}( X, G)$. Then $\varphi^* \omega$ extends to
  a holomorphic $1$--form on $A$ by Hartog's theorem. As any
  holomorphic $1$--form on $A$ is closed, it follows that $\omega$
  is closed. Since $H^1( \Bun^{d, \regstab}( X, G), \CC) = 0$ by
  \cite{AB}, we conclude $\omega = df$ for a holomorphic function
  $f$ on $\Bun^{d, \regstab}( X, G)$. But any such function $f$ is
  constant, so $\omega  = 0$.
\end{proof}
We denote by $\Higgs^{d, \regstab}(X, G) \subseteq \Higgs^d(X, G)$
the open locus of Higgs $G$--bundles $( E, \theta)$ for which
$E$ is regularly stable. The forgetful map
\begin{equation} \label{eq:forget_higgs}
  \Higgs^{d, \regstab}( X, G) \longto \Bun^{d, \regstab}( X, G),
    \qquad ( E, \theta) \longmapsto E,
\end{equation}
is an algebraic vector bundle with fibers
$\rmH^0( X, \ad( E) \otimes K_X) \cong \rmH^1( X, \ad( E))^*$,
so it is the cotangent bundle of $\Bun^{d, \regstab}( X, G)$.  
\begin{corollary} \label{cor:higgs}
  The restriction of the algebraic tangent bundle
  \begin{equation*}
    T \Higgs^{d, \regstab}( X, G) \longto
      \Higgs^{d, \regstab}( X, G)
  \end{equation*}
  to the subvariety $\Bun^{d, \regstab}( X, G) \subseteq
  \Higgs^{d, \regstab}( X, G)$ has no nonzero holomorphic sections.
\end{corollary}
\begin{proof}
  The subvariety in question is the zero section
  of the vector bundle \eqref{eq:forget_higgs}.
  Given a vector bundle $V \longto M$ with zero section
  $M \subseteq V$, there is a natural isomorphism
  \begin{equation} \label{eq:decomposition}
    (TV)|_M \cong TM \oplus V
  \end{equation}
  of vector bundles over $M$. In our situation, both summands
  have no nonzero holomorphic sections, according to
  Proposition \ref{prop:fields} and Proposition \ref{prop:forms}.
\end{proof}
Let $\Conn^{d, \regstab}(X, G) \subseteq \Conn^d(X, G)$ denote the
open locus of holomorphic $G$--connections $( E, D)$ for which $E$
is regularly stable.
\begin{proposition} \label{prop:conn}
  There are no holomorphic sections for the forgetful map
  \begin{equation} \label{eq:forget_conn}
    \Conn^{d, \regstab}( X, G) \longto \Bun^{d, \regstab}( X, G),
    \qquad ( E, D) \longmapsto E.
  \end{equation}
\end{proposition}
\begin{proof}
  The map \eqref{eq:forget_conn} is a holomorphic torsor
  under the cotangent bundle of $\Bun^{d, \regstab}( X, G)$.
  As such, it is isomorphic to the torsor of
  holomorphic connections on the line bundle
  \begin{equation*}
    \calL \longto \Bun^{d, \regstab}( X, G)
  \end{equation*}
  with fibers $\det \rmH^1(X, \ad(E))$; see \cite[Lemma IV.4]{Fa}.
  Since $\calL$ is ample, its first Chern class is nonzero,
  so $\calL$ admits no global holomorphic connections.
\end{proof}
Let $\Hod^{d, \regstab}( X, G) \subseteq \Hod^d( X, G)$ denote
the open locus of triples $( \lambda, E, D)$ for which $E$ is
regularly stable. The forgetful maps in \eqref{eq:forget_higgs}
and \eqref{eq:forget_conn} extend to the forgetful map
\begin{equation} \label{eq:forget_hodge}
  \Hod^{d, \regstab}( X, G) \longto \Bun^{d, \regstab}( X, G),
    \qquad ( \lambda, E, D) \longmapsto E,
\end{equation}
which is an algebraic vector bundle. It contains
the cotangent bundle \eqref{eq:forget_higgs} as a subbundle;
the quotient is a line bundle, which is trivialized by
the projection $\pr$ in \eqref{eq:pr_hodge}.
\begin{corollary} \label{cor:conn}
  The vector bundle \eqref{eq:forget_hodge} has no
  nonzero holomorphic sections.
\end{corollary}
\begin{proof}
  Let $s$ be a holomorphic section of the vector bundle
  \eqref{eq:forget_hodge}. Then $\pr \circ s$
  is a holomorphic function on $\Bun^{d, \regstab}( X, G)$,
  and hence constant. This constant vanishes because of
  Proposition \ref{prop:conn}. So $\pr \circ s = 0$, which
  implies that $s = 0$ using Proposition \ref{prop:forms}.
\end{proof}
\begin{corollary} \label{cor:hodge}
  The restriction of the algebraic tangent bundle
  \begin{equation*}
    T \Hod^{d, \regstab}( X, G) \longto \Hod^{d, \regstab}( X, G)
  \end{equation*}
  to the subvariety $\Bun^{d, \regstab}( X, G) \subseteq
  \Hod^{d, \regstab}( X, G)$ has no nonzero holomorphic sections. 
\end{corollary}
\begin{proof}
  Use the decomposition \eqref{eq:decomposition},
  Proposition \ref{prop:fields}, and Corollary \ref{cor:conn}.
\end{proof}

\section{Torelli theorems} \label{sec:torellis}
Let $X, X'$ be compact connected Riemann surfaces of genus $\geq 3$.
Let $G, G'$ be nontrivial connected semisimple linear algebraic
groups over $\CC$. Fix $d \in \pi_1( G)$ and $d' \in \pi_1( G')$.
\begin{theorem} \label{thm:higgs}
  If $\Higgs^{d'}( X', G')$ is biholomorphic to $\Higgs^d( X, G)$,
  then $X' \cong X$.
\end{theorem}
\begin{proof}
  Corollary \ref{cor:higgs} implies that the subvariety
  $\Bun^d( X, G)$ is fixed pointwise by every holomorphic
  $\CC^*$--action on $\Higgs^d( X, G)$. All other complex analytic
  subvarieties with that property have smaller dimension, due to
  Proposition \ref{prop:fixed_higgs}. Thus we get a biholomorphic
  map from $\Bun^{d'}( X', G')$ to $\Bun^d( X, G)$ by restriction.
  Using \cite{BH}, this implies that $X' \,\cong\, X$.
\end{proof}
\begin{theorem} \label{thm:hodge}
  If $\Hod^{d'}( X', G')$ is biholomorphic to $\Hod^d( X, G)$,
  then $X' \cong X$.
\end{theorem}
\begin{proof}
  The argument is exactly the same as in the previous proof.
  It suffices to replace Corollary \ref{cor:higgs} by Corollary
  \ref{cor:hodge}, and Proposition \ref{prop:fixed_higgs}
  by Corollary \ref{cor:fixed_hodge}.
\end{proof}
\begin{theorem} \label{thm:MDH}
  If $\MDH^{d'}( X', G')$ is biholomorphic to $\MDH^d( X, G)$,
  then $X' \cong X$ or $X' \cong \Xbar$.
\end{theorem}
\begin{proof}
  The argument is similar. Corollary \ref{cor:hodge} implies that
  the two subvarieties $\Bun^d( X, G)$ and $\Bun^{-d}( \Xbar, G)$
  are fixed pointwise by every holomorphic $\CC^*$--action on
  $\MDH^d( X, G)$. All other complex analytic subvarieties with
  that property have  smaller dimension, due to Corollary 
  \ref{cor:fixed_MDH}. Thus we get a biholomorphic map from
  $\Bun^{d'}( X', G')$ to either $\Bun^d( X, G)$ or
  $\Bun^{-d}( \Xbar, G)$ by restriction. Using \cite{BH},
  this implies that either $X'\,\cong\, X$ or $X' \,\cong\,\Xbar$.
\end{proof}

\end{document}